\documentclass[11pt]{article}
\usepackage[utf8]{inputenc}
\usepackage{lmodern}
\usepackage{subfiles}
\usepackage{enumitem}
\setenumerate{topsep=6pt,ref={\normalfont(\roman*)},label={\normalfont(\roman*)}, itemsep=0pt} %
\usepackage{pgfplots}
\pgfplotsset{compat=1.18}
\usepgfplotslibrary{groupplots}
\usepackage{amsfonts}
\usepackage{amsthm}
\usepackage{amsmath}
\usepackage{amssymb}
\usepackage{amscd}
\usepackage{mathrsfs}
\usepackage{mathtools}
\usepackage{bbm}
\usepackage{esint}

\usepackage[margin=3cm]{geometry}
\usepackage{setspace}
\usepackage{indentfirst}
\usepackage{graphicx}
\usepackage{graphics}
\usepackage{lscape}
\usepackage{pgf,tikz}
\usepackage{tikz-cd}
\usepackage{color}
\usepackage{pict2e}
\usepackage{epic}
\usepackage{epstopdf}
\usepackage{titlesec, titlefoot}
\titleformat{\section}[block]{\Large\bfseries\filcenter}{\thesection}{1em}{}
\usepackage{commath}
\usepackage{float}
\usepackage{caption}
\usepackage{etoolbox}
\usepackage[affil-it]{authblk}
\usepackage{combelow}

\usepackage[hidelinks, bookmarksdepth=3]{hyperref}
\hypersetup{bookmarksopen=true} 
\usepackage{hypcap}

\graphicspath{{./}}
\allowdisplaybreaks

\expandafter\def\expandafter\normalsize\expandafter{%
\normalsize
\setlength\abovedisplayskip{6pt}
\setlength\belowdisplayskip{6pt}
\setlength\abovedisplayshortskip{6pt}
\setlength\belowdisplayshortskip{6pt}
}

\theoremstyle{plain}

\renewcommand*\thesection{\arabic{section}}
\numberwithin{equation}{section} 

\newtheorem{theorem}{Theorem}[section]
\newtheorem{lemma}[theorem]{Lemma}
\newtheorem*{lemma*}{Lemma}
\newtheorem{proposition}[theorem]{Proposition}

\theoremstyle{definition}

\newtheorem{remark}[theorem]{Remark}

\expandafter\let\expandafter\oldproof\csname\string\proof\endcsname
\let\oldendproof\endproof
\renewenvironment{proof}[1][\proofname]{%
\oldproof[\upshape \bfseries #1]%
}{\oldendproof}

\makeatletter
\def\@makechapterhead#1{%
\vspace*{50\p@}%
{\parindent \z@ \raggedright \normalfont
\interlinepenalty\@M
\Huge\bfseries  \thechapter.\quad #1\par\nobreak
\vskip 40\p@
}}
\makeatother

\newcommand{\eps}{\varepsilon}

\DeclareMathOperator{\cp}{Cap}

\DeclareMathOperator{\diam}{diam}

\DeclareMathOperator{\Lip}{Lip}

\def \Om{\Omega}
\def \R {\mathbb{R}}

\def \T{\mathbb{T}}

\def \p{\partial}
\def \mc{\mathcal}
\def \mb{\mathbb}

\begin{document}

	\title{\textbf{Logarithmic capacity and torsion in planar shape optimization}}
	
	\author[1]{{\Large David Ziener}}
	
	\affil[1]{\small Max Planck Institute for Mathematics in the Sciences, Inselstrasse 22, 04103 Leipzig, Germany
	\protect\\
	{\tt{david.ziener@mis.mpg.de}}\ }
	
	\date{}
	
    \maketitle

	\begin{abstract}
	We study a family of functionals involving a product of logarithmic capacity and torsional rigidity. Using several geometric inequalities, we identify the discs as the only maximizers among planar convex sets for certain values of the parameters. Finally, we prove regularity of maximizers in the non-critical case. This is done by establishing a shape-differentiability theorem for the logarithmic capacity in the class of planar convex bodies.
	\end{abstract}

    \medskip


    \smallskip

	
	\unmarkedfntext{
	\hspace{-0.75cm}
	\emph{Acknowledgments.} 	DZ would like to thank his PhD advisor László Székelyhidi for fruitful discussions regarding this work and Jonas Hirsch for helpful comments on an earlier version of this manuscript.
	}
	
	\vspace{1cm}

    \section{Introduction}
    Consider the variational problem
    \begin{equation}
        \label{eq:abstractShape}
        \sup\{J(\Om):\Omega \in \mc A\},
    \end{equation}
    where $\mc A$ is a class of admissible sets in $\R^n$ and $J:\mc A \to \R$ a shape functional. In this context there are two natural questions. First, does there exist an optimal shape that attains the supremum in \eqref{eq:abstractShape}? Second, if we already know the existence of an optimal shape, is there a way to prove qualitative properties of the set? In the ideal case one can uniquely identify the set, for example, as a ball. In situations where such a unique identification is not possible, one can try to describe the set in a qualitative way, for example through its regularity.\\
    Questions of this form are quite classical and gave rise to a large number of isoperimetric inequalities. The classical isoperimetric inequality characterizes the disc as the unique minimizer of perimeter among planar sets of prescribed area. Its PDE analogue for torsion is Saint-Venant’s inequality \cite{SaintVenant1856Torsion}: among sets of prescribed area, the disc maximizes the torsional rigidity. In planar potential theory, the corresponding isocapacitary inequality states that the disc minimizes logarithmic capacity under an area constraint. Together with the Faber–Krahn inequalities, these results are part of the symmetrization theory developed by Pólya and Szegő \cite{Plya1948TorsionalRP,polya1951isoperimetric}.\\
    Modern shape optimization extends this classical theory in several directions. A central issue is to identify topologies in which minimizing or maximizing sequences are compact and the relevant shape functionals are semicontinuous; general existence results of this kind were established by Buttazzo and Dal Maso and subsequently developed in the monograph of Bucur and Buttazzo \cite{ButtEx,bucur:hal-00414080}. Another direction concerns the regularity of optimal sets, particularly when perimeter terms are present \cite{lamboley2012regularity,DEPHILIPPIS2018147}. Problems involving several geometric quantities are also studied through their Blaschke–Santaló diagrams \cite{lucardesi2022blaschke}, and through the direct optimization of products containing torsional rigidity, capacity, perimeter and volume \cite{https://doi.org/10.1112/blms.12422,briani2021some,van2023some}.\\
    In this paper we are interested in a specific family of shape functionals defined for bounded, open $\Om\subset \R^2$ as
    \[
    H_\alpha(\Om):=\frac{T(\Om)^{1/2}\cp(\overline \Om)}{|\Om|^\alpha P(\Om)^{3-2\alpha}},\quad \alpha \in [0,3/2].
    \]
    Here $T(\Om)$ is the torsional rigidity, $\cp(\overline \Om)$ is the logarithmic capacity and $P(\Om)$ is the perimeter. For the precise definitions we refer to Section 2. This functional was studied in \cite{van2023some}. In \cite{van2023some} the authors considered the variational problem
    \begin{equation}
        \label{eq:IntroVarProb}
        \sup\{H_\alpha(\Om): \Om\subset \R^2 \text{ non-empty, open, bounded, convex}\},
    \end{equation}
    and proved existence of a maximizer in the regime $\alpha \in [0,3/2)$, see Theorem 7 in \cite{van2023some}. For the critical case $\alpha=3/2$, existence of a maximizer was shown among sets which are suitably close in the sense of a volume constraint to being an ellipse, see Theorem 5 in \cite{van2023some}.\\
    The goal of this paper is to further understand the properties of maximizers of \eqref{eq:IntroVarProb}. We do so by two different approaches. \\
    First, in Section 3, we use some known global geometric inequalities to prove existence of a maximizer for \eqref{eq:IntroVarProb} in the critical case $\alpha=3/2$. The loss of compactness, which is a priori present in this variational problem, is overcome through the Makai inequality, Proposition \ref{prop:Makai}. In Section 4, again by using global geometric inequalities, we identify a regime of parameters $\alpha \in [0,\alpha^*]$ for some $\alpha^*>1$, where the disc is the unique maximizer.\\
    Finally, in Section 5, we prove in the regime $\alpha \in [0,3/2)$ the $C^{1,1}$ regularity of maximizers. This is done by employing an abstract result of Lamboley, Novruzi and Pierre \cite{lamboley2012regularity}. To use this abstract result, we prove that the logarithmic potential energy, or equivalently the logarithmic capacity, is $C^1$ regular as a shape functional on Lipschitz perturbations. Moreover, we express the derivative through the density of the equilibrium measure for the logarithmic potential energy. These results extend the higher-dimensional analogues (see for instance \cite{Jerison,crasta2005long}) and might be of independent interest. 
    
    \section{Preliminaries}
    We denote the open unit disc as
    \[
    \mb D:=\{x\in \R^2:|x|<1\}.
    \]
    Let $\Om \subset \R^2$ be bounded and open. The \textit{torsion function} $u_\Om \in H^1_0(\Om)$ is the weak solution of
    \[
        -\Delta u_\Om=1,\quad \text{in }\Om.
    \]
    The \textit{torsional rigidity} or \textit{torsion} is defined as
    \[
    T(\Om)=\int_{\Om}u_\Om.
    \]
    We denote for $t>0$ the set $t\Om:=\{tx:x\in\Om\}$. Then one has the scaling
    \begin{equation}
        \label{scalingTorsion}
        T(t\Om)=t^4T(\Om),\quad t>0.
    \end{equation}
    There are two central inequalities for the torsional rigidity. The first one was conjectured by Saint-Venant in \cite{SaintVenant1856Torsion} and shown by Pólya in \cite{Plya1948TorsionalRP}.
    \begin{proposition}[Saint-Venant]
        \label{prop:SaintVenant}
        Let $\Om \subset \mb R^2$ be bounded and open. Denote by $\Om^*$ the disc with $|\Om^*|=|\Om|$, then $T(\Om)\leq T(\Om^*)$.
    \end{proposition}
    The second inequality, which is due to Makai \cite{Makai1962}, needs a convexity assumption.
    \begin{proposition}[Makai]
        \label{prop:Makai}
        Suppose $\Om \subset \R^2$ is bounded, open, and convex, then
        \[
        T(\Om)\leq \frac23 \frac{|\Om|^3}{P(\Om)^2}.
        \]
    \end{proposition}
    Let $K \subset \R^2$ be compact. We define the \textit{logarithmic potential energy} as
    \[
    I(K):=\inf_{\mu \in \mc P(K)}\int_{K\times K}\log(|x-y|^{-1})d\mu(x)d\mu(y).
    \]
    A set $K$ is called \textit{polar} if $I(K)=\infty$ and \textit{nonpolar} if $I(K)<\infty$. The \textit{logarithmic capacity} of $K$ is defined as
    \[
    \cp(K):=\exp(-I(K)).
    \]
    The scaling for the logarithmic capacity is
    \begin{equation}
        \label{scalingLogCap}
        \cp(tK)=t\cp(K),\quad t>0.
    \end{equation}
    We recall two geometric inequalities. The first one is classical. See for example \cite{duren2013menahem}.
    \begin{proposition}
        \label{prop:Polya}
    Let $\Om \subset \R^2$ be a Jordan domain with area $|\Om|=\pi$. Then $\cp(\overline \Om)\geq 1$ and it holds
    \begin{equation}
        \label{ineq:PolyaClass}
        \cp(\overline\Om)\leq \frac{P(\Om)}{2\pi}.
    \end{equation}
    Equality in \eqref{ineq:PolyaClass} holds if and only if $\Om=\mb D$ up to rigid motions.
    \end{proposition}
    The second inequality is a sharpened version of Proposition \ref{prop:Polya} and was recently shown in \cite[Lemma 4.4]{niebel2025globalrigiditytwodimensionalbubbles}.
    \begin{proposition}
    \label{prop:mainInequality}
    Let $\Om \subset \R^2$ be a Jordan domain with area $|\Om|=\pi$. Then it holds
    \begin{equation}
        \label{ineq:SharpPol}
        4\cp(\overline \Om)\leq \frac{P(\Om)}{2\pi}+\sqrt{5\left(\frac{P(\Om)}{2\pi}\right)^2+4}.
    \end{equation}
    Equality in \eqref{ineq:SharpPol} holds if and only if $\Om=\mb D$ up to rigid motions.
    \end{proposition}
    Finally, we recall the logarithmic capacity of a line segment: For $a,b \in \R$ one has
    \begin{equation}
        \label{eq:CapofLine}
        \cp([a,b]\times \{0\})=\frac{|b-a|}{4},
    \end{equation}
    see for example \cite[p. 335]{pommerenke1975univalent}.
    
    \section{The functional $H$}
    We introduce the functional
    \[
    H(\Om):=\frac{T(\Om)^{1/2}\cp(\overline \Om)}{|\Om|^{3/2}}, \quad \text{for }\Om \subset \R^2,\text{ bounded, open and nonempty.}
    \]
    Note that $H$ is invariant under replacing $\Om$ by $t\Om$.\\
    As mentioned in the introduction, our goal is to study planar shapes which maximize the functional $H=H_{\alpha}$, with $\alpha=3/2$. The first result is the observation that the supremum is finite among convex planar sets.
    \begin{proposition}
        \label{prop:FiniteSupH}
        It holds that
        \[
        \sup\{H(\Om): \Om \text{ non-empty, open, bounded, convex}\}\leq  \frac{1+2\sqrt 2}{2\sqrt 3} H(\mb D).
        \]
    \end{proposition}

    \begin{proof}
        By scale invariance we can assume without loss of generality that $|\Omega|=\pi$. Set
        \[
        p=\frac{P(\Omega)}{2\pi}.
        \]
        Then from Proposition \ref{prop:Polya} $p\ge 1$, and
        \[
        \frac{H(\Omega)}{H(\mb D)}
        =
        \left(\frac{T(\Omega)}{T(\mb D)}\right)^{1/2}\cp(\overline{\Om}).
        \]
        
        Combining Proposition \ref{prop:SaintVenant} and Proposition \ref{prop:Makai} we have
        \[
        \left(\frac{T(\Omega)}{T(\mb D)}\right)^{1/2}
        \le 
        \min\left\{1,\frac{2}{p\sqrt{3}}\right\}
        \]
        From Proposition \ref{prop:mainInequality} we have
        \[
        \cp(\overline \Om)\le c(p):=\frac{p+\sqrt{5p^{2}+4}}{4}.
        \]
        Therefore
        \[
        \frac{H(\Omega)}{H(\mb D)}
        \le 
        c(p)\min\left\{1,\frac{2}{p\sqrt{3}}\right\}
        =:F(p).
        \]
        
        Thus, we have to maximize $F$ for $p\ge 1$. If $1\le p\le 2/\sqrt{3}$, then
        \[
        F(p)=c(p),
        \]
        and $c(p)$ is increasing, since
        \[
        c'(p)=\frac14\left(1+\frac{5p}{\sqrt{5p^2+4}}\right)>0.
        \]
        Hence on $[1,2/\sqrt{3}]$, $F$ is maximized at $p=2/\sqrt{3}$.
        
        If $p\ge 2/\sqrt{3}$, then
        \[
        F(p)=
        \frac{1}{2\sqrt{3}}\left(1+\sqrt{5+\frac{4}{p^2}}\right),
        \]
        which is decreasing in $p$. Hence on $[2/\sqrt{3},\infty)$, $F$ is also maximized at $p=2/\sqrt{3}$. Thus,
        \[
        F(p)\le F(2/\sqrt{3})=c(2/\sqrt{3})=\frac{1+2\sqrt{2}}{2\sqrt{3}},
        \]
        which completes the proof.
    \end{proof}

    \begin{remark}
        In \cite[Theorem 4]{van2023some} the same statement is shown with constant $8$ instead of $\frac{1+2\sqrt 2}{2\sqrt 3}$. Since our constant is
        \[
        \frac{1+2\sqrt 2}{2\sqrt 3}\approx1.11
        \]
        much closer to the value of $1$, which we conjecture here to be the optimal bound, we decided to include the proof.
    \end{remark}
    We denote in the following
    \[
        \mc C_\pi:=\{\Om \subset \R^2 :\Om \text{ non-empty, open, bounded, convex, }|\Om|=\pi\}.
    \]
    The main result of this section is the following theorem.
    \begin{theorem}
        \label{thm:ExistenceOfMaximizer}
        There exists a non-empty, open, bounded, convex maximizer of the variational problem
        \[
        \sup\{H(\Om): \Om \text{ non-empty, open, bounded, convex}\}.
        \]
    \end{theorem}
    We first give a simple consequence of Proposition \ref{prop:Makai}.
    \begin{lemma}
        \label{lem:consequenceOfMakai}
        If $\Om \subset \R^2$ is open, bounded and convex, then
        \[
        T(\Om)\leq \frac16 \frac{|\Om|^3}{\diam(\Om)^2}.
        \]
    \end{lemma}
    \begin{proof}
        The inequality follows immediately from Proposition \ref{prop:Makai} after showing that 
        \begin{equation}
            \label{ineq:DiamPer}
            2\diam(\Om) \leq P(\Om).
        \end{equation}
        Pick $x,y \in \p \Om$ such that
        \begin{equation}
            \label{eq:PointsAttainingd}
            \diam(\Om) =|x-y|.
        \end{equation}
        Since $\Om$ is bounded, open, and convex in $\R^2$, its boundary is a rectifiable Jordan curve. Thus the points $x,y\in\p\Om$ split $\p\Om$
        into two rectifiable arcs $\gamma_1$ and $\gamma_2$, both connecting $x$
        to $y$, such that
        \[
        \partial\Omega=\gamma_1\cup\gamma_2
        \]
        and
        \[
        P(\Omega)=\ell(\gamma_1)+\ell(\gamma_2).
        \]
        For each $i=1,2$, since $\gamma_i$ is a curve connecting $x$ and $y$,
        its length is at least the Euclidean distance between its endpoints. Hence
        \[
        |x-y|\leq \ell(\gamma_i), \qquad i=1,2.
        \]
        Therefore, by \eqref{eq:PointsAttainingd}
        \[
        2\diam(\Omega)
        =2|x-y|
        \leq \ell(\gamma_1)+\ell(\gamma_2)
        =P(\Omega).
        \]
        This proves \eqref{ineq:DiamPer} and hence the lemma.
    \end{proof}
    Next we note that the logarithmic capacity is continuous with respect to Hausdorff convergence among planar connected sets. This is a well-known fact, but we did not find an exact reference, so we provide the simple proof here.
    \begin{lemma}
        \label{lem:ContinuityOfLogCap}
        Suppose $K_n\subset \R^2$ is a sequence of compact, connected sets with $\diam(K_n)>0$. Suppose that
        \[
        d_H(K_n,K)\to 0,
        \]
        where $K \subset \R^2$ is some compact set with $\mathrm{diam}(K)>0$ and $d_H$ is the Hausdorff distance. Then
        \[
        \cp(K_n)\to\cp(K).
        \]
    \end{lemma}
    \begin{proof}
        This follows immediately from Theorem 3.3 in \cite{kalmykov2022continuity}, after proving that every compact, connected set $E\subset \R^2$ with positive diameter satisfies
        \begin{equation}
            \label{eq:uniformlyperfect}
            E\cap \left\{z:\frac{1}{2} r \leq |z-a|\leq r\right\}\neq \emptyset,\quad \forall0< r\leq \mathrm{diam}(E),a\in E.
        \end{equation}
        In the language of \cite{kalmykov2022continuity} this means the set $E$ is $1/2$-uniformly perfect. Indeed, let $0<r\leq \diam E$, $a \in E$ and pick $x,y \in E$ such that
        \[
        |x-y|=\diam (E).
        \]
        Without loss of generality, we can assume that
        \begin{equation}
            \label{eq:lowerboundxa}
            |x-a|\geq \frac{\diam (E)}{2}\geq \frac r2.
        \end{equation}
        The map $g:E \to \R$, $g(z)=|z-a|$ is continuous and $g(E)\subset \R$ is an interval containing $0$ and by \eqref{eq:lowerboundxa} it contains $r/2$. In particular
        \[
        E\cap \left\{z:\frac{1}{2} r \leq |z-a|\leq r\right\}
        \]
        cannot be empty, which proves \eqref{eq:uniformlyperfect}.
    \end{proof}
    \begin{lemma}
        \label{lem:RatBound}
        Consider a sequence of sets $(\Om_n)_n \subset \mc C_\pi$. If $\diam(\Om_n)\to \infty$, then up to taking a subsequence
        \[
        \frac{\cp(\overline \Om_n)}{\diam(\Om_n)}\to \frac14.
        \]
    \end{lemma}

    \begin{proof}
        Let
        \[
        K_n:=p_n+t_n\overline \Om_n,
        \]
        where $t_n>0$ and $p_n \in \R^2$ are chosen such that $\diam(K_n)=1$ and $0 \in K_n$. Note that
        \[
        |K_n|\to 0,
        \]
        by the volume constraint $|\Om_n|=\pi$. In particular, after taking subsequences, Blaschke's selection theorem yields a convex limit $K$, which has $\mathrm{diam}(K)=1$ and $|K| =0$. This forces $K$ to be a line segment of length $1$, without loss of generality say
        \[
        K=[0,1] \times\{0\}.
        \]
        All sets we work with have positive diameter and are convex, hence connected. Therefore, from Lemma \ref{lem:ContinuityOfLogCap} and by continuity of diameter under Hausdorff convergence of compact sets, we have
        \[
        \frac{\cp(\overline \Om_n)}{\diam(\Om_n)}=\frac{\cp(K_n)}{\diam(K_n)}\to\frac{\cp(K)}{\diam(K)},
        \]
        where we used that the ratio $\frac{\cp(\cdot)}{\diam(\cdot)}$ is invariant under translating and rescaling by \eqref{scalingLogCap} in the first equality. Since
        \[
        \cp([0,1]\times \{0\})=\frac{1}{4},
        \]
        by \eqref{eq:CapofLine}, the proof follows.
    \end{proof}

    \begin{proof}[Proof of Theorem \ref{thm:ExistenceOfMaximizer}]
        Note by scale-invariance of $H$, we can restrict to sets in $\mc C_\pi$.\\ We proceed in two steps, first we rule out the regime of thin sets, second we conclude by the direct method.\\
        \textbf{Step 1: Ruling out degeneracy.} First note that
        \[
        M:=\sup_{\Om\in \mc C_\pi}H(\Om)<\infty,
        \]
        by Proposition \ref{prop:FiniteSupH}. Hence, we can consider a maximizing sequence of sets $(\Om_n)_n \subset \mc C_\pi$ such that  
        \[
        H(\Om_n)\to M.
        \]
        Suppose for the sake of contradiction that $\diam(\Om_n) \to \infty$. Then from Lemma \ref{lem:RatBound} we have 
        \begin{equation}
            \label{ineq:MaxContra}
             H(\Om_n)\leq \frac{\cp(\overline \Om_n)}{\sqrt{6}\diam(\Om_n)}\to\frac{1}{4\sqrt{6}}<\frac{1}{2\sqrt{2}\pi}=H(\mb D),
        \end{equation}
        where we used that $T(\mb D)=\pi/8$, see \cite[Equation (9)]{van2023some} and $\cp(\mb D)=1$ in the last equality. The inequality \eqref{ineq:MaxContra} is a contradiction to the fact that the sequence is maximizing.\\
        \textbf{Step 2: Compactness.} From the above we can assume that up to a subsequence $\sup_{n}
        \diam(\Om_n)<\infty$. This means in particular that (after translating the sets if necessary), there exists some set $\Om \in \mc C_\pi$, such that up to a subsequence
        \begin{equation}
            \label{eq:ConvToMax}
            d_H(\overline \Om_n,\overline \Om)\to 0.
        \end{equation}
        By Lemma \ref{lem:ContinuityOfLogCap}, the logarithmic capacity is continuous with respect to \eqref{eq:ConvToMax}. The same statement for the torsional rigidity is given in \cite[Theorem 3.8]{colesanti2010minkowski}. Hence, we have by the choice of $(\Om_n)_n$ as a maximizing sequence that
        \[
        H(\Om)=\lim_{n\to \infty}H(\Om_n)=M,
        \]
        so $\Om$ is the desired maximizer.
    \end{proof}

    \section{The functional $H_\alpha$}
    In this section we consider a version of the functional $H$. Define
    \[
    H_\alpha(\Om):=\frac{T(\Om)^{1/2}\cp(\Om)}{|\Om|^\alpha P(\Om)^{3-2\alpha}},\quad \alpha \in [0,3/2].
    \]
    We consider the functional on two classes of sets. First, let us consider Jordan domains for which we can prove a sharp result.
    \begin{theorem}
        \label{thm:HalphaJordan}
        The following holds.
        \begin{enumerate}
        \item \label{it:halpha1}
        For $0\leq \alpha \leq 1$ we have 
        \[
        \sup\{H_\alpha(\Om): \Om \text{ non-empty, Jordan domain}\}=H_\alpha(\mb D),
        \]
        and $\mb D$ is the unique maximizer up to translations and dilations.
        \item \label{it:halpha2}
        For $1<\alpha \leq \frac32$ one has
        \[
        \sup\{H_\alpha(\Om): \Om \text{ non-empty, Jordan domain}\}=\infty.
        \]
    \end{enumerate}
    \end{theorem}

    \begin{proof}
        We start with the proof of \ref{it:halpha1}. Let $\Om$ be a non-empty Jordan domain and suppose without loss of generality that $|\Om|=\pi$. Denote $p:=\frac{P(\Om)}{2\pi}$. As $0\leq \alpha \leq 1$, by using first Proposition \ref{prop:Polya} and then Proposition \ref{prop:SaintVenant} we obtain immediately
        \begin{equation}
            \label{eq:HalphaUp}
            \frac{H_\alpha(\Om)}{H_\alpha(\mb D)}=\left(\frac{T(\Om)}{T(\mb D)}\right)^{1/2}\frac{\cp(\overline \Om)}{p^{3-2\alpha}} \leq \left(\frac{T(\Om)}{T(\mb D)}\right)^{1/2}p^{2\alpha-2}\leq p^{2\alpha-2}\leq 1.
        \end{equation}
        Suppose equality holds in \eqref{eq:HalphaUp}. Then in particular we have equality in Proposition \ref{prop:Polya}, which implies that $\Om =\mb D$, up to rigid motions.\\
        For the proof of \ref{it:halpha2} we construct a sequence of planar Jordan domains $\Om_n$ such that
        \[
        H_\alpha(\Om_n)\to \infty,\quad \text{as }n \to \infty.
        \]
        For $n\in \mb N$ sufficiently large, consider the domain
        \[
        \Om_n=\mb D\cup \big((-n,n)\times(-n^{-2},n^{-2})\big).
        \]
        By monotonicity of the torsion we have
        \[
        T(\Om_n)\geq T(\mb D).
        \]
        Similarly, we have by \eqref{eq:CapofLine} that
        \[
        \cp(\overline \Om_n)\geq \cp([-n,n]\times \{0\})=\frac n2.
        \]
        Next, we can bound the area from above as
        \[
        |\Om_n|\leq  \pi +4 nn^{-2}=\pi+\frac{4}{n}\leq C,
        \]
        where $C$ is independent of $n$. Finally, it is clear that for some $C>0$ independent of $n$ one has
        \[
        P(\Om_n)\leq Cn.
        \]
        Hence, the estimates above imply
        \[
        H_\alpha(\Om_n)\geq C n^{2\alpha-2}
        \]
        which goes to $\infty$ since $\alpha>1$.
    \end{proof}
    In the convex class, exactly such degeneracies are ruled out. We first note the following.
    \begin{proposition}
        \label{prop:existenceMaxHalpha}
        Let $\alpha \in [0,3/2]$. There exists a non-empty, open, bounded, convex maximizer of the variational problem
        \[
        \sup\{H_\alpha(\Om): \Om \text{ non-empty, open, bounded, convex}\}.
        \]
    \end{proposition}

    \begin{proof}
        The case $0\leq \alpha < 3/2$ is shown in \cite[Theorem 7 (iii)]{van2023some}. The case $\alpha =3/2$ is exactly Theorem \ref{thm:ExistenceOfMaximizer}.
    \end{proof}

    In view of Theorem \ref{thm:HalphaJordan} one could think that at least for disc maximality among convex sets the threshold $\alpha=1$ could be critical. This is disproven in the next theorem.
    \begin{theorem}
        There exists $1<\alpha^*<\frac32$ such that
        \[
        \sup\{H_\alpha(\Om): \Om \text{ non-empty, open, bounded, convex}\}=H_\alpha(\mb D),\quad \alpha \in [0,\alpha^*],
        \]
        and the disc is the unique maximizer up to translations and dilations.
    \end{theorem}

    \begin{proof}
        We pick $\alpha^*=9/8$. Let $\Om \in \mc C_\pi$. We denote $p:=\frac{P(\Om)}{2\pi}$ and $\beta=3-2\alpha$. From the normalization one has
        \begin{equation}
            \label{eq:RatioEstbeta}
            \frac{H_\alpha(\Om)}{H_\alpha(\mb D)}=\frac{H_{3/2}(\Om)}{H_{3/2}(\mb D)}p^{-\beta}.
        \end{equation}
        From the proof of Proposition \ref{prop:FiniteSupH} we know that for $1\leq p\leq \frac{2}{\sqrt3}$
        \begin{equation}
            \label{finiteUpperBound}
            \frac{H_{3/2}(\Om)}{H_{3/2}(\mb D)}\leq c(p),\quad c(p):=\frac{p+\sqrt{5p^2+4}}{4}.
        \end{equation}
        A simple exercise in calculus shows that
        \begin{equation}
            \label{ineq:GrowthBound}
            c(p)<p^{3/4},\quad \text{for }1< p\leq \frac{2}{\sqrt 3},
        \end{equation}
        with equality only at $p=1$. Let $0\leq \alpha \leq \alpha^*$. Then we have by the choice of $\alpha^*$ that $\beta \geq 3/4$. We obtain therefore:
        \[
        \frac{H_\alpha(\Om)}{H_\alpha(\mb D)}\leq c(p)p^{-\beta}\leq c(p)p^{-3/4}.
        \]
        If $1\leq p \leq 2/\sqrt{3}$ we can use \eqref{ineq:GrowthBound} directly to deduce the desired statement. Hence, suppose $p \geq 2/\sqrt{3}$. Let us denote $F$ as in the proof as Proposition \ref{prop:FiniteSupH}. Using \eqref{eq:RatioEstbeta} and following the proof of Proposition \ref{prop:FiniteSupH}, we estimate
        \begin{align}
            \frac{H_\alpha(\Om)}{H_\alpha(\mb D)}\leq F(p)p^{-\beta}\leq c\left(\frac{2}{\sqrt{3}}\right)\left(\frac{2}{\sqrt{3}}\right)^{-3/4}\leq 1,
        \end{align}
        using \eqref{ineq:GrowthBound} in the final inequality. If equality holds in the estimates above, we necessarily have that $p=1$. As $|\Om|=\pi$, this means that equality holds in the isoperimetric inequality which forces $\Om=\mb D$, up to rigid motions.
    \end{proof}

    \begin{remark}
        It is clear that the choice of $\alpha^*$ in the proof above is not optimal. If one optimizes the upper bound obtained from Propositions \ref{prop:SaintVenant}, \ref{prop:Makai} and \ref{prop:mainInequality}, one can reach a value of
        \[
        \alpha^*\approx1.152.
        \]
        Since this is still far from the desired value of $3/2$ we give the simpler proof above.
    \end{remark}
    
    \section{Regularity of Maximizers for $H_\alpha$}
    The main result of this section is the following theorem.
    \begin{theorem}
        \label{thm:HalphaReg}
        Let $0\leq\alpha<\frac32$. Let $\Om \subset \R^2$ be a bounded, open, convex maximizer for
        \[
        \sup\{H_\alpha(\Om): \Om \text{ non-empty, open, bounded, convex}\},
        \]
        which exists by Proposition \ref{prop:existenceMaxHalpha}. Then $ \p \Om$ is $C^{1,1}$ regular.
    \end{theorem}

    \subsection{Notation}
    We introduce some notation. First, since we are dealing with the logarithmic potential energy, we work with compact sets instead of open sets as before. Fix a nonpolar compact set $E \subset \R^2$. Denote the space
    \[
    X:=W^{1,\infty}(E;\R^2),\quad ||h||_X:=||h||_{L^\infty}+\Lip(h;E), \quad U:=\{h \in X:\Lip(h;E)< 1\}.
    \]
    For $h \in U$, denote the map $\Phi_h:E\to \Phi_h(E)\subset \R^2$, $\Phi_h(x):=x+h(x)$ and $E_h:=\Phi_h(E)$. Moreover, let
    \[
    I(h):=I(E_h)=\inf_{\mu \in \mc P(E_h)}\int_{E_h\times E_h}\log(|x-y|^{-1})d\mu(x)d\mu(y).
    \]
    We can pull back the variational problem which defines the logarithmic potential energy as follows. Denote the kernel
    \[
    K_h(x,y):=-\log|x+h(x)-(y+h(y))|.
    \]
    As $h \in U$, the map $\Phi_h:E\to E_h$ is bi-Lipschitz. Therefore, every probability measure $\tilde \mu \in \mc P(E_h)$ can be written uniquely as $\tilde \mu =(\Phi_h)_*\mu$ for $\mu \in \mc P(E)$. This implies in particular that
    \[
    I(h)=\min_{\mu \in \mc P(E)} J_h(\mu),\quad J_h(\mu):=\int\int K_h(x,y)d\mu(x)d\mu(y).
    \]
    We denote a minimizer for $J_h$ by $\mu_h$. This means by definition, that $(\Phi_h)_*\mu_h$ is the equilibrium measure of $E_h$. Finally, we introduce for $h\in U$, $q \in X$ the kernel
    \[
    B_{h,q}(x,y):=-\frac{(\Phi_h(x)-\Phi_h(y))\cdot (q(x)-q(y))}{|\Phi_h(x)-\Phi_h(y)|^2},\quad x,y \in E,\;x\neq y.
    \]
    \subsection{Differentiability of Logarithmic Capacity}
    
    \begin{proposition}
        \label{prop:C1Diff}
        The map $h\mapsto I(h)$ is $C^1$ on the open set $U$. For $h \in U$ and $q \in X$ we have
        \begin{equation}
            \label{eq:GeneralKernFormDiff}
            DI(h)[q]=-\int\int_{E\times E}\frac{(\Phi_h(x)-\Phi_h(y))\cdot(q(x)-q(y))}{|\Phi_h(x)-\Phi_h(y)|^2}d\mu_h(x)d\mu_h(y).
        \end{equation}
    \end{proposition}

    \begin{remark}
        For compact sets with sufficiently regular boundary and under smooth normal perturbations, the differentiability of logarithmic capacity, or equivalently logarithmic potential energy, follows from Hadamard’s variational formula for the Green function with pole at infinity; see \cite[Section 3, formula (31)]{SchifferHadamard}.\\
        First variation formulas are also available without any boundary smoothness beyond convexity for Minkowski perturbations, expressed in terms of support functions; in the planar Laplacian case, see \cite[Proposition 10.1]{AkmanLewisSaariVogel+2021+247+302}. Such formulas, however, concern one-parameter variations within the convex class and do not provide the $C^1$-differentiability under general $W^{1,\infty}$-deformations needed below to apply Theorem \ref{thm:Lamboley}. Proposition \ref{prop:C1Diff} establishes this stronger property.
    \end{remark}

    \begin{proof}[Proof of Proposition \ref{prop:C1Diff}]
        Fix $h \in U$. Suppose $q_n \to 0$ in $X$. To show the differentiability we want to prove that
        \begin{equation}
            \label{eq:C1LogCap}
            I(h+q_n)-I(h)-DI(h)[q_n]=o(||q_n||_X),
        \end{equation}
        where $DI(h)[q_n]$ is given by \eqref{eq:GeneralKernFormDiff}.\\
        We write
        \[
        \mu_n:=\mu_{h+q_n},\quad \mu:=\mu_h.
        \]
        By Lemma \ref{lem:KernelExpansion} below, we have
        \[
        I(h+q_n)-I(h)\leq \int\int(K_{h+q_n}-K_h)d\mu d\mu=\int\int B_{h,q_n}d\mu d\mu+O(\Lip(q_n)^2).
        \]
        Hence
        \begin{equation}
            \label{eq:UpperBoundC1}
            I(h+q_n)-I(h)-DI(h)[q_n]\leq o(||q_n||_X).
        \end{equation}
        For the reverse inequality we use first the definition of $\mu_n$ and $\mu$ and then Lemma \ref{lem:KernelExpansion} to estimate
        \begin{align}
            \label{ineq:LowerboundMoving}
            I(h+q_n)-I(h)&=J_{h+q_n}(\mu_n)-I(h) \nonumber \\
            &\geq J_{h+q_n}(\mu_n)-J_h(\mu_n) \nonumber \\
            &=\int\int(K_{h+q_n}-K_h)d\mu_nd\mu_n \nonumber \\
            &\geq\int\int B_{h,q_n}d\mu_n d\mu_n-C\Lip(q_n)^2.
        \end{align}
        Note that from Lemma \ref{lem:MovingConvergence} below applied with $g_n :=h$, $\tilde h_n =h+q_n$ and $k=q_n/||q_n||_X$, it follows
        \begin{equation}
            \label{eq:RemovingMoving}
            \int\int B_{h,q_n}d\mu_nd\mu_n=\int\int B_{h,q_n}d\mu d\mu+o(||q_n||_X)=DI(h)[q_n]+o(||q_n||_X).
        \end{equation}
        Combining \eqref{ineq:LowerboundMoving} and \eqref{eq:RemovingMoving} we obtain
        \begin{equation}
            \label{eq:LowerBoundC1}
            I(h+q_n)-I(h)-DI(h)[q_n]\geq -o(||q_n||_X).
        \end{equation}
        The above inequality together with \eqref{eq:UpperBoundC1} proves differentiability.\\
        Next, we prove the continuity of
        \[
        DI:U \to X^*,\quad h \mapsto DI(h).
        \]
        Fix $h \in U$. Since $\Lip(h;E)\leq \alpha<1$, we have $\lambda:=1-\alpha>0$. Suppose $h_n \to h$ in $X$. Denote $\eta_n :=\mu_{h_n}\otimes \mu_{h_n}$ and $\eta:=\mu_{h}\otimes \mu_{h}$. We have for $k \in X$ with $||k||_X\leq 1$
        \begin{align*}
            DI(h_n)[k]-DI(h)[k]&=\left(\int B_{h_n,k}d\eta_n-\int B_{h_n,k}d\eta\right)+\int(B_{h_n,k}-B_{h,k})d\eta\\
            &=:A_n(k)+B_n(k).
        \end{align*}
        From Lemma \ref{lem:MovingConvergence} applied with $g_n=\tilde h_n=h_n$, it follows
        \begin{equation}
            \label{unifConvA}
            \sup_{||k||_X\leq 1}|A_n(k)|\to 0,\quad \text{as }n \to \infty.
        \end{equation}
        For $B_n(k)$ note that by definition of $B_{h_n,k}$ one obtains
        \begin{equation}
            \label{BnLipschitzBound}
            |B_{h_n,k}(x,y)-B_{h,k}(x,y)|\leq C_\lambda||h_n-h||_X.
        \end{equation}
        Indeed, by Lemma \ref{lem:KernelExpansion},
        \[
        K_{h_n+||h_n-h||_Xk}-K_{h_n}
        =
        ||h_n-h||_X B_{h_n,k}
        +
        O_\lambda(||h_n-h||_X^2).
        \]
        On the other hand, by Lemma \ref{lem:KernelExpansion} and the linearity of $B_{h,\cdot}$,
        \[
        \begin{aligned}
        K_{h_n+||h_n-h||_Xk}-K_{h_n}
        &=
        (K_{h+(h_n-h)+||h_n-h||_Xk}-K_h)
        -(K_{h+(h_n-h)}-K_h)
        \\
        &=
        B_{h,(h_n-h)+||h_n-h||_Xk}
        -B_{h,h_n-h}
        +
        O_\lambda(||h_n-h||_X^2)
        \\
        &=
        ||h_n-h||_X B_{h,k}
        +
        O_\lambda(||h_n-h||_X^2).
        \end{aligned}
        \]
        Therefore,
        \[
        ||h_n-h||_X
        |B_{h_n,k}-B_{h,k}|
        \leq
        C_\lambda ||h_n-h||_X^2,
        \]
        and dividing the above by $||h_n-h||_X$ gives \eqref{BnLipschitzBound}.\\
        Hence from \eqref{BnLipschitzBound},
        \begin{equation}
            \label{unifConvB}
            \sup_{||k||_X\leq 1}|B_n(k)|\leq C_\lambda ||h_n-h||_X\to 0,\quad \text{as }n \to \infty.
        \end{equation}
        Combining \eqref{unifConvA} and \eqref{unifConvB} we have shown
        \[
        \sup_{||k||_X\leq 1}| DI(h_n)[k]-DI(h)[k]|\to 0, \quad \text{as }n \to \infty
        \]
        which is exactly the desired continuity.
    \end{proof}

    \begin{lemma}
        \label{lem:KernelExpansion}
        Suppose $h \in X$ with $\Lip(h)\leq \alpha <1$. Write $\lambda:=1-\alpha >0$. For each $q \in X$ with $\Lip(q)$ small enough one has the expansion
        \[
        K_{h+q}(x,y)-K_h(x,y)=B_{h,q}(x,y)+R_{h,q}(x,y),
        \]
        where
        \[
        |B_{h,q}|\leq C_\lambda \Lip(q),\quad |R_{h,q}|\leq C_\lambda\Lip(q)^2.
        \]
    \end{lemma}

    \begin{proof}
        By assumption on $h$ we have that $|\Phi_h(x)-\Phi_h(y)|\geq \lambda|x-y|$. Hence
        \begin{equation}
            \label{ineq:RatioBound}
            \frac{|q(x)-q(y)|}{|\Phi_h(x)-\Phi_h(y)|}\leq \lambda^{-1}\Lip(q).
        \end{equation}
        We write
        \[
        K_{h+q}(x,y)-K_h(x,y)=-\log|e+a|,\quad e=\frac{\Phi_h(x)-\Phi_h(y)}{|\Phi_h(x)-\Phi_h(y)|},\quad a:=\frac{q(x)-q(y)}{|\Phi_h(x)-\Phi_h(y)|}.
        \]
        Note that for $\Lip(q)$ small enough we have from \eqref{ineq:RatioBound} that $|a|<1/2$.\\
        The map $a\mapsto -\log|e+a|$ is in $C^2$ on $|a|<1/2$ with bounds uniform in $e \in S^1$. Hence we have the Taylor expansion at $a=0$ as
        \[
        -\log|e+a|=-e\cdot a+O(|a|^2)=-\frac{(\Phi_h(x)-\Phi_h(y))\cdot(q(x)-q(y))}{|\Phi_h(x)-\Phi_h(y)|^2}+O\left(\frac{|q(x)-q(y)|^2}{|\Phi_h(x)-\Phi_h(y)|^2}\right).
        \]
        The claim follows easily from \eqref{ineq:RatioBound}.
    \end{proof}

    \begin{lemma}
    \label{lem:MovingConvergence}
    Let $g_n\to h$ and $\tilde h_n\to h$ in $X$, and assume that $g_n,\tilde h_n,h$ all lie eventually in a common uniformly bi-Lipschitz neighborhood. Let
    \[
        \mu_n:=\mu_{\tilde h_n},\qquad \mu:=\mu_h,
    \]
    and set
    \[
        \eta_n:=\mu_n\otimes\mu_n,\qquad \eta:=\mu\otimes\mu.
    \]
    Then
    \begin{equation*}
        \sup_{\|k\|_X\leq 1}
        \left|
            \int\int B_{g_n,k}\,d\mu_n d\mu_n
            -
            \int\int B_{g_n,k}\,d\mu d\mu
        \right|
        \to 0,\quad \text{as }n \to \infty.
    \end{equation*}
\end{lemma}

\begin{proof}
    We divide the proof into two steps. First we prove the convergence of measures, then we prove that this is enough to pass to the limit uniformly in the kernel integral.\\
    \textbf{Step 1: Convergence of measures.} We want to prove here that
    \begin{equation}
        \label{eq:ConvergenceOfMeas}
        \eta_n \rightharpoonup \eta.
    \end{equation}
    Since $\tilde h_n\to h$ in $X$ and all maps lie eventually in a common uniformly bi-Lipschitz neighborhood, integrating Lemma \ref{lem:KernelExpansion} gives
    \begin{equation}
        \label{eq:JBound}
        |J_{\tilde h_n}(\nu)-J_h(\nu)|
        \leq C\|\tilde h_n-h\|_X
    \end{equation}
    for the atomless probability measures used below. Hence, using the minimality of $\mu_n$ for $J_{\tilde h_n}$, we estimate
    \begin{align}
        \label{eq:MinSeq}
        J_h(\mu_n)
        &\leq J_{\tilde h_n}(\mu_n)+C\|\tilde h_n-h\|_X  \notag\\
        &\leq J_{\tilde h_n}(\mu)+C\|\tilde h_n-h\|_X \notag\\
        &\leq J_h(\mu)+2C\|\tilde h_n-h\|_X .
    \end{align}
    In particular, \eqref{eq:MinSeq} says that $(\mu_n)$ is a minimizing sequence for $J_h$. Up to taking subsequences, $\mu_n\rightharpoonup \tilde\mu$. By lower semicontinuity of the logarithmic kernel, $\tilde\mu$ is minimizing for $J_h$. By uniqueness it follows that $\tilde\mu=\mu$. Since the convergence holds for all subsequences, we have shown that $\mu_n\rightharpoonup\mu$. The convergence \eqref{eq:ConvergenceOfMeas} follows by compactness of $E$.\\
    \textbf{Step 2: Uniform convergence of kernels.} We want to prove
    \begin{equation}
        \label{eq:KernelConv}
        \sup_{\|k\|_X\leq 1}
        \left|
            \int\int B_{g_n,k}\,d\mu_n d\mu_n
            -
            \int\int B_{g_n,k}\,d\mu d\mu
        \right|
        \to 0.
    \end{equation}
    Since all the maps lie in a common uniformly bi-Lipschitz neighborhood, there exists $\lambda>0$ such that
    \[
        |\Phi_{g_n}(x)-\Phi_{g_n}(y)|\geq \lambda |x-y|
    \]
    for all $n$ sufficiently large and all $x,y\in E$. Hence as in the proof of Lemma \ref{lem:KernelExpansion}
    \begin{equation}
        \label{eq:KernelUniformBound}
        |B_{g_n,k}(x,y)|
        \leq C_\lambda,
        \qquad x\neq y,\quad \|k\|_X\leq 1.
    \end{equation}
    Fix $\eps>0$. Since $\eta$ gives no mass to the diagonal, we can choose $\delta>0$ such that
    \begin{equation*}
        \eta(F_{2\delta})<\eps,
        \qquad
        F_\rho:=\{(x,y)\in E\times E: |x-y|\leq \rho\}.
    \end{equation*}
    From \eqref{eq:ConvergenceOfMeas} and because $F_{2\delta}$ is closed, we have
    \begin{equation}
        \label{eq:LimSupSmall}
        \limsup_{n\to\infty}\eta_n(F_{2\delta})
        \leq \eta(F_{2\delta})
        <\eps.
    \end{equation}
    Let $\chi_\delta$ be a smooth cutoff such that $\chi_\delta\equiv 0$ on $F_\delta$ and $\chi_\delta\equiv 1$ on $E\times E\setminus F_{2\delta}$. We claim that the family
    \[
        \mathcal F_\delta
        :=
        \{\chi_\delta B_{g_n,k}: n\in\mathbb N,\ \|k\|_X\leq 1\}
    \]
    is equicontinuous and uniformly bounded on $E\times E$. Uniform boundedness follows immediately from \eqref{eq:KernelUniformBound}. For equicontinuity close to the diagonal, the claim follows from the cutoff. In the set $E\times E\setminus F_\delta$, note that
    \[
        (x,y,a,b)\mapsto
        -\frac{(\Phi_{g_n}(x)-\Phi_{g_n}(y))\cdot(a-b)}
        {|\Phi_{g_n}(x)-\Phi_{g_n}(y)|^2}
    \]
    is uniformly Lipschitz, because $g_n\to h$ in $X$ and we stay a positive distance away from the diagonal. Moreover, the maps $k$ with $\|k\|_X\leq 1$ are uniformly bounded and uniformly Lipschitz. This yields the claim.\\
    Hence, from Arzelà-Ascoli, the family $\mathcal F_\delta$ is totally bounded in $C(E\times E)$. Thus, we can pick a finite set $\{\psi_1,\ldots,\psi_N\}\subset C(E\times E)$ such that for every $n$ and every $k$ with $\|k\|_X\leq 1$, there is an index $j(n,k)\in\{1,\ldots,N\}$ satisfying
    \[
        \|\chi_\delta B_{g_n,k}-\psi_{j(n,k)}\|_{L^\infty}<\eps.
    \]
    From the convergence \eqref{eq:ConvergenceOfMeas}, we have
    \[
        \max_{1\leq j\leq N}
        \left|
            \int \psi_j\,d\eta_n
            -
            \int \psi_j\,d\eta
        \right|
        \to 0.
    \]
    Hence,
    \[
        \sup_{\|k\|_X\leq 1}
        \left|
            \int \chi_\delta B_{g_n,k}\,d\eta_n
            -
            \int \chi_\delta B_{g_n,k}\,d\eta
        \right|
        \leq 2\eps+o(1).
    \]
    Finally, using the properties of the cutoff, the uniform bound \eqref{eq:KernelUniformBound}, and \eqref{eq:LimSupSmall}, we get
    \begin{align*}
        &\sup_{\|k\|_X\leq 1}
        \left|
            \int B_{g_n,k}\,d\eta_n
            -
            \int B_{g_n,k}\,d\eta
        \right| \\
        &\leq
        \sup_{\|k\|_X\leq 1}
        \left|
            \int \chi_\delta B_{g_n,k}\,d\eta_n
            -
            \int \chi_\delta B_{g_n,k}\,d\eta
        \right| \\
        &\qquad
        +
        \sup_{\|k\|_X\leq 1}
        \left|
            \int (1-\chi_\delta) B_{g_n,k}\,d\eta_n
            -
            \int (1-\chi_\delta) B_{g_n,k}\,d\eta
        \right| \\
        &\leq
        2\eps+o(1)
        +
        C_\lambda\bigl(\eta_n(F_{2\delta})+\eta(F_{2\delta})\bigr) \\
        &\leq
        2\eps+C_\lambda\eps+o(1).
    \end{align*}
    Since $\eps>0$ was arbitrary, this completes the proof of \eqref{eq:KernelConv} and hence the proof of the lemma.
\end{proof}
    In the case of a convex set $E$, we can simplify the formula for $DI(h)$ by using  \cite[Theorem 3.1]{goldman2018minimizers}:
    \begin{lemma}
        \label{lem:NovagaRuffiniConvex}
        Let $E$ be a compact convex body. Let $\mu_E$ be the equilibrium measure associated with $E$. Then there exists some $p=p(E)>2$ and $f_E \in L^p(\p E)$ such that $\mu_E=f_E \mc H^1 \llcorner \p E$.
    \end{lemma}

    \begin{lemma}
        \label{lem:FormulaAtConvexSet}
        Let $E$ be a compact convex body, and denote by $\nu$ its outer unit normal. Using the same notation as above, we have for any $k \in W^{1,\infty}(E,\mb R^2)$
        \[
        DI(0)[k]= -2\pi\int_{\p E} f_E^2 k\cdot \nu d\mc H^1.
        \]
    \end{lemma}

    \begin{proof}
        Recall from Proposition \ref{prop:C1Diff} that
        \begin{equation}
            \label{form:DifferentialAt0}
            DI(0)[k]=-\int\int\frac{(x-y)\cdot (k(x)-k(y))}{|x-y|^2}d\mu_E(x)d\mu_E(y).
        \end{equation}
        From Lemma \ref{lem:NovagaRuffiniConvex} for some $p>2$, the equilibrium measure has the form
        \begin{equation}
            \label{equiMeasurePf}
            \mu_E=f_E\mc H^1\llcorner\p E,\quad f_E \in L^p(\p E).
        \end{equation}
        For $\eps>0$ we introduce
        \[
        A_\eps(x)=\int_{\p E\cap\{|x-y|>\eps\}}\frac{x-y}{|x-y|^2}f_E(y)d\mc H^1(y).
        \]
        We compute
        \begin{align*}
            &\int\int_{ (E \times  E) \cap \{|x-y|>\eps\}}\frac{(x-y)\cdot (k(x)-k(y))}{|x-y|^2}d\mu_E(x)d\mu_E(y) \\
            &=\int\int_{ (E \times  E) \cap \{|x-y|>\eps\}}\frac{(x-y)}{|x-y|^2} \cdot k(x) d\mu_E(x)d\mu_E(y) \\
            &- \int\int_{(E \times E) \cap \{|x-y|>\eps\}}\frac{(x-y)}{|x-y|^2} \cdot k(y) d\mu_E(x)d\mu_E(y) \\
            &=2\int_{\p E}A_\eps(x)\cdot k(x) f_E(x) d\mc H^1(x),
        \end{align*}
        where we used the antisymmetry of the kernel and the formula \eqref{equiMeasurePf} in the last equality. In particular,
        \begin{equation}
            \label{eq:Diff0KernelForm}
            DI(0)[k]=-2\lim_{\eps\downarrow 0}\int_{\p E}A_\eps(x)\cdot k(x) f_E(x)d\mc H^1(x).
        \end{equation}
        Since $\p E$ is Lipschitz and $f_E \in L^p(\p E)$ where $p>2$, we have by the Coifman–McIntosh–Meyer Theorem \cite{coifman1982integrale} and standard Calderón-Zygmund estimates that
        \[
        A(x):=\lim_{\eps \downarrow 0}\int_{\p E \cap \{|x-y|>\eps\}}\frac{x-y}{|x-y|^2}f_E(y) d\mc H^1(y),\quad \text{exists for $\mc H^1$ a.e. }x\in \p E
        \]
        and
        \begin{equation}
            \label{eq:pConvOfAeps}
            A_\eps \to A,\quad \text{in }L^p(\p E).
        \end{equation}
        Now introduce the potential $v_E$ which we define as
        \[
        v_E(x)=\int_{\p E}\log\left(\frac{1}{|x-y|}\right)f_E(y)d\mc H^1(y).
        \]
        Hence
        \[
        -\Delta v_E = 2\pi \mu_E,\quad \text{in the sense of distributions.}
        \]
        By Frostman's Theorem (see \cite[Theorem 3.3.4]{ransford1995potential}) it follows that $v_E\equiv I(E)$ in $E$ up to polar sets. In particular, since $\mu_E$ is supported on $\p E$ we have that 
        \begin{equation}
            \label{eq:ConstancyOfLogPot}
            v_E\equiv I(E),\quad \text{in }\mathrm{int}(E).
        \end{equation}
        By the jump formulas for the single layer potential (see \cite[Theorem 1.1]{tolsa2020jump} for example) we have that
        \begin{equation}
            \label{eq:jumpForm}
            (\nabla v_E)^{\pm}(x)=-A(x) \mp\pi f_E(x)\nu(x),\quad \text{for $\mc H^1$ a.e. }x\in \p E,
        \end{equation}
        where $(\nabla v_E)^-$ denotes the interior and $(\nabla v_E)^+$ the exterior trace. From \eqref{eq:ConstancyOfLogPot} it follows that $(\nabla v_E)^-=0$ and therefore \eqref{eq:jumpForm} yields
        \begin{equation}
            \label{eq:FormulaForA}
            A=\pi f_E\nu,\quad \text{a.e. on }\p E.
        \end{equation}
        Finally, by the convergence \eqref{eq:pConvOfAeps}, and since we have $f_E k \in L^p \subset L^{p'}$ for $p>2$ we can use Hölder to pass to the limit
        \begin{equation*}
            \lim_{\eps\downarrow 0}\int_{\p E}A_\eps(x)\cdot k(x) f_E(x)d\mc H^1(x)=\int_{\p E}A(x)\cdot k(x)f_E(x)d\mc H^1(x).
        \end{equation*}
        Combining \eqref{eq:FormulaForA} with \eqref{eq:Diff0KernelForm} we obtain
        \[
        DI(0)[k]=-2\pi \int_{\p E}f_E^2 k \cdot \nu d\mc H^1,
        \]
        which completes the proof.
    \end{proof}

    \subsection{Proof of the regularity theorem}
    Let $\mb T:=\R/2\pi\mb Z$. Given a compact convex set $E\subset \R^2$ after a possible translation, we can associate as in \cite{lamboley2012regularity} a positive function $u \in W^{1,\infty}(\mb T)$, such that
    \[
    E=E_u:=\left\{re^{i\theta}: 0\leq r\leq \frac{1}{u(\theta)}\right\}.
    \]
    It is well-known that
    \[
    E_u \text{ is convex} \iff u''+u\geq 0 \quad \text{in the sense of distributions}.
    \]
    With the formulation above, we can work with functions instead of sets when dealing with problems of planar shape optimization.\\
    The central abstract ingredient is \cite[Theorem 2.4]{lamboley2012regularity}, which we state here for the convenience of the reader. \\
    Let $u_0 \in W^{1,\infty}(\T)$, with $u_0>0$ a minimizer of
    \[
    \min\{j(u): u \in W^{1,\infty}(\mb T), u>0,u''+u\geq 0\}.
    \]
    \begin{theorem}[\cite{lamboley2012regularity}]
        \label{thm:Lamboley}
        Assume that the functional $j$ is given as
        \[
        j(u)=r(u)+\int_{\mb T}G(\theta,u(\theta),u'(\theta))d\theta,
        \]
        with the following properties.
        \begin{enumerate}
            \item $r:W^{1,\infty}(\mb T)\to \R$ is $C^1$ near $u_0$.
            \item $G(\theta,u,q)$ is $C^2$ near $\mb \T \times u_0(\mb T)\times \mathrm{Conv}(u_0'(\mb T))$, where $\mathrm{Conv}(u_0'(\T))$ is the smallest closed interval containing the values of the right- and left-derivatives $u_0'(\theta^+)$, $u_0'(\theta^-)$ for all $\theta \in \mb T$.
            \item For some $p \in [1,\infty]$ the derivative $r'(u_0)$ is represented by an $L^p$ density that is
            \[
            r'(u_0)v =\int_{\mb T}\rho(\theta)v(\theta)d\theta,\quad \rho \in L^p(\mb T).
            \]
            \item $G_{qq}(\theta,u,q)>0$ on the relevant neighborhood of $\mb \T \times u_0(\mb T)\times \mathrm{Conv}(u_0'(\mb T))$.
        \end{enumerate}
        Then
        \[
        u_0\in W^{2,p}(\mb T).
        \]
    \end{theorem}

    \begin{proof}[Proof of Theorem \ref{thm:HalphaReg}]
        The proof reduces to the verification of the assumptions in Theorem \ref{thm:Lamboley}. In the following fix $\alpha \in [0,3/2)$.\\
        \textbf{Step 1: Setup.} Recall the gauge formulation and write sets as
        \[
        E_u=\left\{(r,\theta): 0\leq r\leq \frac{1}{u(\theta)}\right\}.
        \]
        Suppose $E_{u_0}$ is the maximizer for $H_\alpha$ among planar convex sets. We use the equality $\log \cp(E_u)=-I(E_u)$ to denote
        \begin{align*}
            j(u)&=-\log H_\alpha(E_u) \\
            &=(3-2\alpha)\log P(E_u)+\alpha\log |E_u|-\frac{1}{2}\log T(E_u)+I(E_u)\\
            &=r(u)+ \int_{\T}G(u(\theta),u'(\theta))d\theta,
        \end{align*}
        where we define
        \begin{equation}
            \label{eq:FormulaForR}
            r(u):= (3-2\alpha)\left(\log P(E_u)-\frac{P(E_u)}{P(E_{u_0})}\right)+\alpha\log |E_u|-\frac{1}{2}\log T(E_u)+I(E_u),
        \end{equation}
        and
        \[
        G(u,q):=\frac{3-2\alpha}{P(E_{u_0})}\frac{\sqrt{u^2+q^2}}{u^2}.
        \]
        Note that
        \[
        \frac{(3-2\alpha)}{P(E_{u_0})}P(E_u)=\int_{\T}G(u(\theta),u'(\theta))d\theta,
        \]
        see for example equation (26) in \cite{lamboley2012regularity}.
        Moreover, one easily checks that
        \[
        G_{qq}(u,q)= \frac{3-2\alpha}{P(E_{u_0})}\frac{1}{(u^2+q^2)^{3/2}}>0,
        \]
        as $\alpha <3/2$.\\
        \textbf{Step 2: Integrability.} We want to show that there exists $s>1$ such that $r'(u_0)\in L^s(\mb T)$. We note first that $r$ is $C^1$ in a neighborhood of $u_0$. For all terms in \eqref{eq:FormulaForR} except $I(E_u)$ this is well-known: For the torsional rigidity see \cite[Proposition 3.13]{lamboley2012regularity}. For the area term it follows easily from equation (26) in \cite{lamboley2012regularity}. Finally, for $u \mapsto I(E_u)$ this is Proposition \ref{prop:C1Diff} combined with a standard construction of a deformation $u \mapsto \xi(u)$, see equations (27)-(29) in \cite{lamboley2012regularity}. Moreover, note that the perimeter does not contribute to $r'(u_0)$.\\
        Let $v \in W^{1,\infty}(\mb T)$. Denote
        \[
        x_\theta=\frac{1}{u_0(\theta)}e^{i\theta},\quad k(x_\theta)=-\frac{v(\theta)}{u_0(\theta)^2}e^{i\theta}.
        \]
        By Lemma \ref{lem:FormulaAtConvexSet} we have
        \begin{equation}
            DI(0)[k]=-2\pi\int_{\mb T}f_{E_0}^2k\cdot \nu d\mc H^1=2\pi\int_{\mb T}f_{E_0}(x_\theta)^2\frac{v(\theta)}{u_0(\theta)^3}d\theta.
        \end{equation}
        Following \cite[Proposition 3.13]{lamboley2012regularity} we have that
        \[
        r'(u_0)v=\int_{\mb T}\rho(\theta)v(\theta)d\theta,
        \]
        where
        \begin{equation}
            \label{formulaRho}
            \rho(\theta):=-\frac{\alpha}{|E_{u_0}|u_0(\theta)^3}+\frac{|\nabla w_0(x_\theta)|^2}{2 T(E_{u_0})u_0(\theta)^3}+2\pi \frac{f_{E_{u_0}}(x_\theta)^2}{u_0(\theta)^3}.
        \end{equation}
        In \eqref{formulaRho}, $w_0$ denotes the torsion function for $E_{u_0}$. In particular, since the first two terms above are in $L^\infty(\mb T)$ and the density $f_{E_0}(x_\theta)$ is in $L^p(\mb T)$ for some $p>2$ we have that $\rho \in L^{s}(\mb T)$ for $s:=\frac p2>1$. This shows that $r'(u_0)$ has the desired integrability.\\
        \textbf{Step 3: Hölder regularity.} By Theorem \ref{thm:Lamboley} it follows that $u_0 \in W^{2,s}(\mb T)$. From the Sobolev embedding
        \[
        W^{2,s}(\mb T)\hookrightarrow C^{1,1-\frac1s}(\mb T),
        \]
        hence $\p E_{u_0} \in  C^{1,1-\frac1s}$.\\
        \textbf{Step 4: Bootstrap.} To complete the proof of $C^{1,1}$ regularity, one can observe that as soon as $\p E_{u_0}\in C^{1,\beta}$ for some $\beta \in (0,1)$ one obtains that $f_{E_{u_0}} \in L^\infty(\mb T)$ see \cite[Lemma 3.6]{goldman2018minimizers}. In particular, from \eqref{formulaRho} we have that $\rho \in L^\infty(\mb T)$. Then applying again Theorem \ref{thm:Lamboley} with $p=\infty$ we have $u_0\in C^{1,1}(\mb T)$ and hence $\p E_{u_0} \in C^{1,1}$.
    \end{proof}

    \begin{remark}
        In a similar way, we can give an alternative proof of Theorem 1.2 in \cite{goldman2018minimizers}. The authors study minimizers of the planar functional
        \[
        F_Q(E):=P(E)+Q^2I(E),
        \]
        among planar convex sets $E$ under a volume constraint and prove their $C^{1,1}$-regularity. Their proof is based on a geometric cap-cut construction. Since Proposition \ref{prop:C1Diff} and Lemma \ref{lem:FormulaAtConvexSet} provide exactly the missing ingredients for the framework of Theorem \ref{thm:Lamboley}, one can avoid the geometric construction of their proof.
    \end{remark}

    \medskip

    \noindent\textbf{AI Disclosure.}
    ChatGPT 5.6 Sol was used to review previous drafts of this
    work for typos and mathematical mistakes, and to assist with numerical experiments. The author reviewed and edited the content as needed and takes full responsibility for the content of this work.

	\let\oldthebibliography\thebibliography
	\let\endoldthebibliography\endthebibliography
	\renewenvironment{thebibliography}[1]{
	\begin{oldthebibliography}{#1}
	\setlength{\itemsep}{0.5pt}
	\setlength{\parskip}{0.5pt}
	}
	{
	\end{oldthebibliography}
	}

\end{document}